\documentclass[12pt]{article}
\usepackage{amsmath,amsopn,amssymb,amsthm,amsfonts}
\usepackage[T2A]{fontenc}
\usepackage[cp1251]{inputenc}
\usepackage{longtable}

\newtheorem{theorem}{Theorem}[section]
\newtheorem{corollary}[theorem]{Corollary}
\newtheorem{proposition}[theorem]{Proposition}
\newtheorem{definition}[theorem]{Definition}
\newtheorem{lemma}[theorem]{Lemma}

\newcommand\g{{\mathfrak g}}
\newcommand\h{{\mathfrak h}}
\newcommand{\R}{\mathbb{R}}

\begin{document}
{\bf \large
\centerline{N.~K.~Smolentsev}

\vspace{3mm}
\centerline{Invariant pseudo-Sasakian and $K$-contact structures on}
\centerline{seven-dimensional nilpotent Lie groups}}
\vspace{3mm}

\begin{abstract}
We study the question of the existence of left-invariant Sasaki contact structures on the seven-dimensional nilpotent Lie groups. It is shown that the only Lie group allowing Sasaki structure  with a positive definite metric tensor is the Heisenberg group. We find a complete list of the 22 classes of seven-dimensional nilpotent Lie groups which admit pseudo-Sasaki structure. We also present a list of 25 classes of seven-dimensional nilpotent Lie groups admitting a $K$-contact structure, but not the pseudo-Sasaki structure. All the contact structures considered are central extensions of six-dimensional nilpotent symplectic Lie groups and are established formulas that connect the geometrical characteristics of the six-dimensional nilpotent almost pseudo-K\"{a}hler Lie groups and seven-dimensional nilpotent contact Lie groups. It is known that for the six-dimensional nilpotent pseudo-K\"{a}hler Lie groups the Ricci tensor is always zero. Unlike the pseudo-K\"{a}hlerian case, it is shown that on contact seven-dimensional algebras the Ricci tensor is nonzero even in directions of the contact distribution.  \end{abstract}

\section{Introduction} \label{Introduction}
A left-invariant \emph{K\"{a}hler} structure on a Lie group $H$ is a triple $(h,\omega,J)$, consisting of a left-invariant Riemannian metric $h$, a left-invariant symplectic form $\omega$ and an orthogonal left-invariant complex structure $J$, where $h(X,Y) = \omega(X,JY)$ for any left-invariant vector fields $X$ and $Y$ on $H$.
Therefore, such a structure on the group $H$ can be defined as a pair $(\omega, J)$, where $\omega$ is the symplectic form, and $J$ is a complex structure compatible with $\omega$, i.e., such that $\omega(JX,JY) =\omega(X,Y)$.
If $\omega(X,JX) > 0$, $\forall X \neq 0$, it becomes a K\"{a}hler metric, and if the positivity condition is not met, then $h(X,Y) = \omega(X,JY)$ is a pseudo-Riemannian metric.
Then, $(h, \omega, J)$ is called the \emph{pseudo-K\"{a}hler structure} on the Lie group $H$.
The left-invariance of these objects implies that the (pseudo) K\"{a}hler structure $(h,\omega,J)$ can be defined by the values of the $\omega$, $J$ and $h$ on the Lie algebra $\h$ of the Lie group $H$. Then, the $(\h, \omega, J, h)$ is called a \emph{pseudo-K\"{a}hler Lie algebra}.
Conversely, if $(\h, J, h)$ is a Lie algebra endowed with a complex structure $J$, orthogonal with respect to the pseudo-Riemannian metric $h$, then the equality $\omega(X,Y) = h(JX,Y)$ determines the (fundamental) two-form $\omega$, which is closed if and only if the $J$ is parallel \cite{KN}.

Classification of the six-dimensional real nilpotent Lie algebras admitting invariant complex structures and estimation of the dimensions of moduli spaces of such structures is obtained in \cite{Sal-1}. In article \cite{Goze-Khakim-Med}, classification is obtained of symplectic structures on six-dimensional nilpotent Lie algebras.
The condition of existence of left-invariant positive definite K\"{a}hler metric on the Lie group $H$ imposes strong restrictions on the structure of its Lie algebra $\h$. For example, Benson and Gordon have shown \cite{BG} that this Lie algebra cannot be nilpotent except in the abelian case. The nilpotent Lie groups and nilmanifolds (except tora) do not allow the left-invariant K\"{a}hler metrics, however,  such manifolds may exist as a pseudo-left-invariant K\"{a}hler metric. The article in \cite{CFU2} provides a complete list of the six-dimensional nilpotent Lie algebras admitting pseudo-K\"{a}hler structures. A more complete study of the properties of the curvature of pseudo-K\"{a}hler structures is carried out in \cite{Smolen-11}.

The analogues of symplectic structures in an odd case are contact structures \cite{Blair}.
It is known \cite{Goze-Khakim-Med} that the contact Lie algebra $\g$ 
with nontrivial center
is a central extension of symplectic Lie algebras $(\h,\omega)$ with the help of a non-degenerate cocycle $\omega$.
In this case, a contact Lie algebra $\g$ admits Sasakian structure only if the Lie algebra $(\h,\omega)$ admits a K\"{a}hler metric.
Therefore, the question of the existence of Sasakian structures on the seven-dimensional nilpotent contact Lie algebra $\g$ is reduced to the question of the existence of K\"{a}hler structures on six-dimensional nilpotent symplectic Lie algebras $\h = \g/Z(\g)$, where $Z(\g)$ is the center of contact Lie algebra $\g$.
The classification of seven-dimensional nilpotent contact Lie algebras is obtained in \cite{Kutsak}.
The authors of \cite{Cappelletti-14} found examples of $K$-contact but not Sasakian structures on the seven-dimensional nilpotent contact Lie algebras. Just as in the even-dimensional nilpotent case, there are topological obstructions to the existence of Sasakian structures on nilpotent Lie algebra, based on the strong Lefschetz theorem for contact manifolds \cite{Cappelletti-13}.

In this paper, we show that on seven-dimensional nilpotent Lie algebras the invariant contact Sasakian structures exist only in the case of the Heisenberg algebra, and pseudo-Sasakian (i.e., having a pseudo-Riemannian metric tensor) there are only 22 classes of central extensions of six-dimensional pseudo K\"{a}hler nilpotent Lie algebras. This article provides a comprehensive list of such pseudo-Sasakian structures and their curvature properties are investigated. We also obtain the list of 25 classes of seven-dimensional nilpotent contact Lie algebras that allow $K$-contact structure, but do not allow pseudo-Sasakian structures.

\section{Preliminaries} \label{Preface}
The basic concepts of contact manifolds and contact Lie algebra are summarized here.

\subsection{Contact manifolds} \label{Cont-Manif}
In this section, we recall some basic definitions and properties in contact Riemannian geometry. For further details we refer the reader to the monographs in \cite{Blair}.
Let $M$ be a smooth manifold of dimension $2n+1$. A 1-form $\eta $ on $M$ is called a \emph{contact form} if $\eta \wedge (d\eta)^n \ne 0$ is a volume form. Then the pair $(M,\eta)$ is called a (strict) \emph{contact manifold}. In any contact manifold one proves the existence of a unique vector field $\xi$, called the \emph{Reeb vector field}, satisfying the properties $\eta(\xi) =1$ and $d\eta(\xi,X) =0$ for all vector fields $X$ on $M$.

The contact form $\eta$ on $M$ determines the distribution of $D =\{X \in TM\, |\, \eta(X) = 0\}$ of dimension $2n$, which is called the \emph{contact distribution}. It is easy to see that $L_\xi\eta = 0$. If $M$ is a contact manifold with a contact form $\eta$, then the \emph{contact metric structure} is called the quadruple $(\eta, \xi, \phi, g)$, where $\xi$  is the Reeb field, $g$ is the Riemannian metric on $M$ and $\phi$ is the affinor on $M$, and for which there are the following properties \cite{Blair}:
\begin{enumerate}
  \item $\phi^2 = -I +\eta\otimes \xi$;
  \item $d\eta(X,Y) = g(X,\phi Y)$;
  \item $g(\phi X,\phi Y) = g(X,Y) -\eta(X)\eta(Y)$,
\end{enumerate}
where $I$ is the identity endomorphism of the tangent bundle. The Riemannian metric $g$ is called the contact metric structure \emph{associated} with the contact structure $\eta$. From the third property it follows immediately that the associated metric $g$ for the contact structure $\eta$ is completely determined by the affinor $\phi$: $g(X,Y) =d\eta(\phi X,Y) + \eta(X)\eta(Y)$.

The contact metric manifold whose Reeb vector field $\xi$ is Killing, $L_\xi g =0$, is called \emph{$K$-contact} \cite{Blair}. This last property is equivalent to the condition $L_\xi \phi =0$. Recall from \cite{Blair} that an \emph{almost contact} structure on a manifold $M$ is a triple $(\eta, \xi, \phi)$, where $\eta$ is a 1-form, $\xi$ is a vector field and $\phi$ is an affinor on $M$ with properties: $\eta(\xi) =1$ and $\phi^2 = -I +\eta\otimes \xi$.

Suppose that $M$ is an almost contact manifold. Consider the direct product $M\times \R$. A vector field on $M\times \R$ is represented in the form $(X, f{d}/{dt})$, where $X$ is the tangent vector to $M$, $t$ is the coordinate on $\R$ and $f$ is a function of class $C^\infty$ on $M\times \R$. We define an almost complex structure $J$ on the direct product $M\times \R$ as follows \cite{Blair}: $J(X, f\,d/dt) = (\phi X - f\xi, \eta(X)\,d/dt)$.

An almost contact structure $(\eta, \xi, \phi)$ is called \emph{normal} if the almost complex structure $J$ is integrable. A four tensor $N^{(1)}$, $N^{(2)}$, $N^{(3)}$ and $N^{(4)}$ is defined \cite{Blair} on an almost contact manifold $M$ by the following expressions:
$$
N^{(1)}(X,Y) = [\phi,\phi](X,Y) +d\eta(X,Y)\xi, \quad  N^{(2)}(X,Y) =(L_{\phi X} \eta)(Y) -(L_{\phi Y} \eta)(X),
$$
$$
N^{(3)}(X,Y) = (L_\xi \phi)X, \quad N^{(4)}(X,Y) = (L_\xi \eta)(X).
$$

An almost contact structure $(\eta, \xi, \phi)$ is a normal, if these tensors vanish. However, it can be shown from the vanishing tensor $N^{(1)}$ that the remaining tensors $N^{(2)}$, $N^{(3)}$ and $N^{(4)}$ also vanish. Therefore, the condition of normality is only the following: $N^{(1)}(X,Y) = [\phi,\phi](X,Y) +d\eta(X,Y)\xi =0$. Thus, a \emph{Sasaki manifold} is a normal contact metric manifold.

\begin{definition}  \label{Pseudo-Sasac}
A pseudo-Riemannian contact metric structure $(\eta, \xi, \phi, g)$ is called a $K$-contact, if the vector field $\xi$  is a Killing vector field. A pseudo-Riemannian contact metric structure $(\eta, \xi, \phi, g)$ is called pseudo-Sasakian if $N^{(1)}(X,Y) = 0$.
\end{definition}

\textbf{Remark.}
In this paper, we assume that the exterior product and exterior differential is determined without a normalizing factor. Then, in particular, $dx\wedge dy = dx\otimes dy -dy\otimes dx$ and $d\eta (X,Y) = X\eta(Y) -Y\eta(X) -\eta([X,Y])$. We also assume that the curvature tensor $R$ is defined by the formula: $R(X,Y)Z = \nabla_X \nabla_Y Z -\nabla_Y \nabla_X Z -\nabla_{[X,Y]} Z$. The Ricci tensor $Ric$ is defined as a contraction of the curvature tensor in the first and the fourth (top) indices: $Ric_{ij}= R_{kij}^k$. In \cite{Blair} Blair shows a different formula for the tensor $N^{(1)}$: $N^{(1)}(X,Y) = [\phi,\phi](X,Y) + 2d\eta (X,Y)\xi$.
In \cite{Blair}, a different formula is used for the exterior derivative $d_0\eta(X,Y) = \frac 12  (X\eta(Y) -Y\eta(X) -\eta([X,Y]))$, which corresponds to the selection of wedge product 1-forms as $dx\wedge dy = \frac 12(dx\otimes dy -dy\otimes dx)$.
This is evident then in determining the associated metrics $g(X,Y) = d\eta (\phi X,Y) +\eta(X)\eta(Y)$.
In our case, the associated Riemannian metric $g$ will be different along the contact distribution $D$ from of the associated Riemannian metric $g_0$ adopted in \cite{Blair}: if $X,Y \in D$, then $g(X,Y) =2g_0(X,Y)$. It further appears that the formula $\nabla_X \xi = -\phi X$ in our case looks like: $\nabla_X \xi = -\frac 12 \phi X$.
In addition in \cite{Blair}, the sectional curvature for the $K$-contact metric manifold in the direction of the 2-plane, containing $\xi$, is equal to 1 (p. 113, Theor. 7.2), and in our case, the sectional curvature is equal to $\frac 14$. Consequently this changes the formula for the Ricci curvature, $Ric(\xi, \xi) = n/2$.

When a manifold is taken Lie group $G$, it is natural to consider the left-invariant contact structure. In this case, the contact form $\eta$, Reeb vector field $\xi$, affinor $\phi$ and associated metric $g$ are defined by its values in the unit $e$, that is, on the Lie algebra $\g$. We will call the $\g$ a \emph{contact Lie algebra} if it is defined in a contact form $\eta\in \g^*$ and the vector $\xi\in \g$, such that $\eta \wedge (d\eta)^n \ne 0$, $\eta(\xi)=1$ and $d\eta(\xi,X)=0$. Note that $d\eta(x,y)=-\eta([x,y])$. In a similar sense, it is considered a \emph{contact metric} Lie algebra, \emph{symplectic} Lie algebra, \emph{pseudo-K\"{a}hler} Lie algebra, \emph{(pseudo) Sasakian} Lie algebra, etc.

Let $\nabla$ be the Levi-Civita connection corresponding to the (pseudo) Riemannian metric $g$. It is determined from the six-membered formula \cite{KN}, which for the left-invariant vector fields $X,Y,Z$ on the Lie group takes the form: $2g(\nabla_X Y,Z) = g([X,Y],Z) +g([Z,X],Y) +g(X,[Z,Y])$.
Recall also that if $R(X,Y)Z$ is the curvature tensor, then the Ricci tensor $Ric(X,Y)$ for the pseudo-Riemannian metric g is defined by the formula:
$$
Ric(X,Y) = \sum_{i=1}^{2n+1} \varepsilon_i g(R(e_i,Y)Z, e_i),
$$
where $\{e_i\}$ is a orthonormal frame on $\g$ and $\varepsilon_i =g(e_i, e_i)$.

\subsection{Contact structures on central extensions} \label{Cont-Alg-Lie}
Contact Lie algebra can be obtained as a result of the \emph{central expansion} of the symplectic Lie algebra $\h$. Recall this procedure. If there is a symplectic Lie algebra $(\h,\omega)$, the central extension $\g = \h \times_\omega\R$ is a Lie algebra in which the Lie brackets are defined as follows: $[X, \xi]_{\g} = 0$, $[X, Y]_{\g} = [X, Y]_{\h} +\omega(X, Y)\xi$ for any $X,Y \in \h$, where $\xi = d/dt$ is the unit vector in $\R$.

On the Lie algebra $\g = \h \times_\omega \R$ contact form given by the form $\eta = \xi^*$, and $\xi = d/dt$ is the Reeb field.
If $x = X + \lambda\xi$ and $y = Y + \mu\xi$, where $X,Y \in \h$, $\lambda, \mu \in \R$, then: $d\eta (x,y) = -\eta([x, y]) = -\xi^*([X,Y]_{\h}+\omega(X,Y)\xi) = -\omega(X,Y)$.

As is known, the isomorphism classes of the central extensions of the Lie algebra $\h$ are in one-to-one correspondence with the elements of $H^2(\h,\R)$. Non-degenerate elements of $H^2(\h,\R)$ (the symplectic Lie algebra) define the contact structures on $\h\times_\omega\R$.

To define the affinor $\phi$ on $\g =\h\times_\omega\R$ we can use an almost complex structure $J$ on $\h$ as follows: if $x =X +\lambda\xi$, where $X\in \h$, then $\phi(x) = JX$.
If this almost complex structure $J$ on $\h$ is also compatible with the $\omega$, that is, it has the property of $\omega(JX,JY) =\omega(X,Y)$, we will get the contact (pseudo) metric structure $(\eta, \xi, \phi, g)$ on $\g =\h\times_\omega\R$, where $g(X,Y) = d\eta(\phi X,Y) + \eta(X)\eta(Y)$. Let $h(X,Y) = \omega(X,JY)$ be a associated (pseudo) Riemannian metric on the symplectic Lie algebra $(\h,\omega)$. Then for $x =X +\lambda\xi$ and $y = Y +\mu\xi$, we have:
$$
g(x,y) = -\omega(JX,Y) + \lambda\mu = h(X,Y) +\lambda\mu.
$$
As is known \cite{KN}, an almost complex structure $J$ is \emph{integrable} (complex), if the Nijenhuis tensor vanishes,

$$
N_J(X,Y) = [JX,JY] -[X,Y] -J[X,JY] -J[JX,Y]=0.
$$
An analogue of the Nijenhuis tensor of the almost complex structure in the case of any tensor field $T$ of type (1,1) is the Nijenhuis torsion \cite{Blair}:
$$
[T,T](X,Y) =T^2[X,Y] +[TX,TY] -T[X,TY] -T[TX,Y].
$$

\begin{proposition}  \label{Prop-1-K-Cont}
A central extension $\g =\h\times_\omega\R$ of almost pseudo-K\"{a}hler Lie algebra $\h$ is $K$-contact Lie algebra.
\end{proposition}

\begin{proof}
As is well known \cite{Blair}, a contact manifold is called $K$-contact if $L_\xi\phi =0$. For left-invariant fields of the form $x =X +\lambda\xi$ and $y = Y +\mu\xi$ where $X,Y\in \h$, we have:
$$
g((L_\xi\phi)x,y) = g(L_\xi(\phi\, x) -\phi(L_\xi x),y) = 0,
$$
because the $L_\xi x = [\xi, X +\lambda\xi] = 0$ and $L_\xi (\phi\, x) =0$. Therefore $L_\xi\phi =0$.
\end{proof}

\begin{proposition}  \label{Prop-2}
Let $(\h,\omega,J)$ be a almost (pseudo) K\"{a}hler Lie algebra $\h$ and $(\eta, \xi, \phi, g)$ be the corresponding contact metric structure on the central extension of $\g =\h\times_\omega\R$. Then the Nijenhuis torsion $[\phi,\phi]$ on $\g$ is expressed in terms of the Nijenhuis tensor $N_J$ almost complex structure $J$ on $\h$ as:
$$
[\phi,\phi](x,y) = N_J(X,Y) -d\eta (x,y)\xi,
$$
where $x =X +\lambda\xi$, $y = Y +\mu\xi$ and $X,Y \in\h$.
\end{proposition}

\begin{proof}
Direct calculations:
\begin{multline*}
[\phi,\phi](x,y) = [\phi,\phi](X +\lambda\xi, Y +\mu\xi) =\\
= \phi^2[X +\lambda\xi, Y +\mu\xi] + [\phi(X +\lambda\xi), \phi(Y +\mu\xi)] -\phi[X +\lambda\xi, \phi(Y +\mu\xi)] - \phi[\phi(X +\lambda\xi), Y +\mu\xi] =\\
= \phi^2([X, Y]_{\h}) +[J(X), J(Y)] -\phi[X +\lambda\xi, J(Y)] -\phi[J(X), Y +\mu\xi] = \\
= -[X, Y]_{\h} +[J(X),J(Y)]_{\h} +\omega(JX, JY)\xi -\phi[X,J(Y)] -\phi[J(X),Y] = \\
= -[X, Y]_{\h} + [J(X), J(Y)]_{\h} + \omega(X, Y)\xi -\phi([X, J(Y)]_{\h} +\omega(X, JY)\xi) -\phi([J(X), Y]_{\h} +\omega(JX, Y)\xi ) = \\
= -[X, Y]_{\h} + [J(X), J(Y)]_{\h} + \omega(X, Y)\xi -J([X, J(Y)]_{\h}) -J([J(X), Y]_{\h}) = \\
= N_J(X,Y) + \omega(X, Y)\xi = N_J(X,Y) -d\eta(x,y)\xi.
\end{multline*}
\end{proof}

\begin{corollary} \label{Cor-1}
The tensor $N^{(1)}(x,y)$ of the contact metric structure $(\eta, \xi, \phi, g)$ on the central expansion $\g =\h\times_\omega\R$ expressed in terms of the Nijenhuis tensor $N_J$ of almost complex structure $J$ on h using the formula:
$$
N^{(1)}(x,y) = N_J(X,Y),
$$
where $x =X +\lambda\xi$, $y = Y +\mu\xi$ and $X,Y \in \h$.
\end{corollary}

\begin{proof}
$N^{(1)}(x,y) = [\phi,\phi](x,y) + d\eta(x,y)\xi = N_J(X,Y) -d\eta(x,y)\xi  +d\eta(x,y)\xi = N_J(X,Y)$.
\end{proof}

\begin{corollary} \label{Cor-2}
Contact metric structure $(\eta, \xi, \phi, g)$ on the central expansion $\g =\h\times_\omega\R$ is (pseudo) Sasakian if and only if the symplectic algebra $(\h,\omega,J)$ is a (pseudo) K\"{a}hler.
\end{corollary}

As mentioned in the introduction, left-invariant K\"{a}hler (i.e. positive-definite) structures do not exist on nilpotent Lie groups other than the torus \cite{BG}. However, the pseudo-K\"{a}hler structure of these Lie groups may exist and in \cite{CFU2} a complete list of these is given. For each Lie algebra on the list in \cite{CFU2} an example is selected of a nilpotent complex structure, and the agreed symplectic forms are found for this. It is more correct to rely on the classification list produced by Goze, Khakimdjanov and Medina \cite{Goze-Khakim-Med}, which lists all symplectic six-dimensional Lie algebras and shows that every nilpotent Lie algebra symplecto-isomorphic is on this list. From this point of view, the author in \cite{Smolen-11} studied the above Lie algebra with the symplectic structure from the list \cite{Goze-Khakim-Med} and found all compatible complex structures. Explicit expressions of complex structures were obtained and the curvature properties of the corresponding pseudo-Riemannian metrics investigated. It was found that there are multi-parameter familyies of complex structures. However, they all share a number of common properties: the associated pseudo K\"{a}hler metric is Ricci-flat, the Riemann tensor has zero pseudo-norm, and the Riemann tensor has a few nonzero components, which depend only on two or, at most, three parameters. Recall also that \cite{Sal-1} presents a classification of the six-dimensional real nilpotent Lie algebras admitting an invariant complex structure and an estimation of the dimension of the moduli spaces of such structures is given.

\textbf{Example}\\
Consider the Lie algebra $\h_{14}$ with commutation relations $[e_1,e_2] = e_4$, $[e_2,e_3] = e_6$, $[e_2,e_4] = e_5$ and symplectic form $\omega = -e^1\wedge e^6 +e^2\wedge e^5 +e^3\wedge e^4$. In \cite{Smolen-11} it is shown that there is a six-parametric family of pseudo-K\"{a}hler metrics on $\h_{14}$, each of which is defined by the operator of the complex structure of the form:
$$
J = \left(\begin{array}{cccccc}
      \psi_{11} & \psi_{12} & 0 & 0 & 0 & 0 \\
      -\frac{\psi_{11}^2+1}{\psi_{12}} & -\psi_{11} & 0 & 0 & 0 & 0 \\
      \frac{\psi_{42}(\psi_{11}^2+1) -2\psi_{41}\psi_{12}\psi_{11}}{\psi_{12}^2} & -\psi_{41} & -\psi_{11} & -\frac{\psi_{11}^2+1}{\psi_{12}} & 0 & 0 \\
      \psi_{41} & \psi_{42} & \psi_{12} & \psi_{11} & 0 & 0 \\
      \psi_{51} & J_{52} & \psi_{42} & \psi_{41} & \psi_{11} & \psi_{12}\\
      \psi_{61} & -\psi_{51}  & -\psi_{41} & \frac{\psi_{42}(\psi_{11}^2+1)-
      2\psi_{41}\psi_{12}\psi_{11}}{\psi_{12}^2} & -\frac{\psi_{11}^2+1}{\psi_{12}}& -\psi_{11} \\
    \end{array}\right),
$$
where $J_{52}= \frac{-2\psi_{11}\psi_{12}(\psi_{42}\psi_{41} -\psi_{12}\psi_{51}) +\psi_{42}^2(\psi_{11}^2 +1) +\psi_{12}^2(\psi_{41}^2 +\psi_{12}\psi_{61})}{(\psi_{11}^2+1)\psi_{12}}$ and $\psi_{12} \ne 0$.
The corresponding pseudo-Riemannian metric is easily obtained in the form $h(X,Y) = \omega(X,JY)$. The curvature tensor is zero for all values of the parameters $\psi_{ij}$.

\section{Contact Lie algebras} \label{Results}
In this section, we will determine when the seven-dimensional nilpotent Lie algebra there are Sasakian and pseudo-Sasakian structure and expresses explicitly the properties of contact structures on a central expansion $\g =\h\times_\omega\R$ through the corresponding properties of the symplectic Lie algebra $(\h,\omega)$.

\begin{theorem} \label{Th-1}
The only seven-dimensional nilpotent Lie algebra admitting a positive definite Sasakian structure is the Heisenberg algebra. A seven-dimensional nilpotent Lie algebra admits pseudo-Sasakian structure if and only if it is one of the algebras on the following list:
\begin{longtable}[H]{|l|l|}
\hline
$\g_{21,1}$ &
\begin{tabular}{l}
   $[e_1,e_2] = e_4, [e_1,e_4] = e_6, [e_2,e_3] = e_6 $, \\
  $[e_1,e_6] = e_7, [e_2,e_5] = e_7, [e_3,e_4] = -e_7$ \\
  $\eta = e^7$, \quad
  $d\eta = -(e^1\wedge e^6 + e^2\wedge e^5 - e^3\wedge e^4)$ \\
\end{tabular} \\
\hline

$\g_{21,2}$ &
\begin{tabular}{l}
   $[e_1,e_2] = e_4$, $[e_1,e_4] = e_6$, $[e_2,e_3] = e_6$, \\
  $[e_1,e_6] = -e_7$, $[e_2,e_5] = -e_7$, $[e_3,e_4] =e_7$ \\
  $\eta = e^7$, \quad
  $d\eta = e^1\wedge e^6 + e^2\wedge e^5 -e^3\wedge e^4$ \\
\end{tabular} \\
\hline

$\g_{14,1}$ &
\begin{tabular}{l}
  $[e_1,e_2] =e_4, [e_2,e_3] =e_6, [e_2,e_4] =e_5$, \\
  $[e_1,e_6] = -e_7, [e_2,e_5] = e_7, [e_3,e_4] = e_7$ \\
  $\eta = e^7$, \quad
$d\eta = -(-e^1\wedge e^6 + e^2\wedge e^5 + e^3\wedge e^4)$ \\
\end{tabular} \\
\hline

$\g_{13,1}$ &
\begin{tabular}{l}
$[e_1,e_2] =e_4, [e_1,e_3] =e_5, [e_1,e_4] =e_6, [e_2,e_3] =e_6$,\\
$[e_1,e_6] = e_7, [e_2,e_5] = \lambda e_7, [e_3,e_4] = (\lambda -1)e_7$	\\
  $\eta = e^7$, \quad
 $d\eta = -(e^1\wedge e^6 + \lambda e^2\wedge e^5 + (\lambda -1) e^3\wedge e^4)$ \\
\end{tabular} \\
\hline

$\g_{13,2}$	&
\begin{tabular}{l}
$[e_1,e_2] =e_4, [e_1,e_3] =e_5, [e_1,e_4] =e_6, [e_2,e_3] =e_6$,\\
$[e_1,e_6] = e_7, [e_2,e_4] = e_7, [e_2,e_5] =  \frac 12 e_7, [e_3,e_4] = -\frac 12 e_7$ \\	
$\eta = e^7$, \quad
 $d\eta = -(e^1\wedge e^6 + e^2\wedge e^4 + \frac 12  e^2\wedge e^5 -\frac 12  e^3\wedge e^4)$ \\
\end{tabular} \\
\hline

$\g_{15,1}$	&
\begin{tabular}{l}
$[e_1,e_2] =e_4, [e_1,e_4] =e_6, [e_2,e_3] =e_5$,\\
$[e_1,e_5] = -e_7, [e_1,e_6] = e_7, [e_2,e_5] = e_7, [e_3,e_4] = e_7$, \\	
$\eta = e^7$, \quad
$d\eta = -(-e^1\wedge e^5 + e^1\wedge e^6 + e^2\wedge e^5 + e^3\wedge e^4)$\\
\end{tabular} \\
\hline

$\g_{15,2}$	&
\begin{tabular}{l}
$[e_1,e_2] =e_4, [e_1,e_3] =e_6, [e_2,e_4] =e_5$,\\
$[e_1,e_5] = e_7, [e_1,e_6] = -e_7, [e_2,e_5] = -e_7, [e_3,e_4] = -e_7$	\\
  $\eta = e^7$, \quad
$d\eta = -e^1\wedge e^5 + e^1\wedge e^6 + e^2\wedge e^5 + e^3\wedge e^4$\\
\end{tabular} \\
\hline

$\g_{11,1}$	&
\begin{tabular}{l}
$[e_1,e_2] =e_4, [e_1,e_4] =e_5, [e_2,e_3] =e_6, [e_2,e_4] =e_6$,\\
$[e_1,e_6] = e_7, [e_2,e_5] = e_7, [e_3,e_4] = -e_7, [e_2,e_6] = \lambda e_7$, \\
  $\eta = e^7$, \quad
$d\eta = -(e^1\wedge e^6 + e^2\wedge e^5 - e^3\wedge e^4 +\lambda e^2\wedge e^6)$
\end{tabular} \\
\hline

$\g_{11,2}$	&
\begin{tabular}{l}
$[e_1,e_2] =e_4, [e_1,e_4] =e_5, [e_2,e_3] =e_6, [e_2,e_4] =e_6$,\\
$[e_1,e_6] = -e_7, [e_2,e_5] = -e_7, [e_3,e_4] = e_7, [e_2,e_6] = -\lambda e_7$,\\
  $\eta = e^7$, \quad
$d\eta = e^1\wedge e^6 + e^2\wedge e^5 - e^3\wedge e^4 + \lambda e^2\wedge e^6$
\end{tabular} \\
\hline

$\g_{10,1}$	&
\begin{tabular}{l}
$[e_1,e_2] =e_4, [e_1,e_4] =e_5, [e_1,e_3] =e_6, [e_2,e_4] =e_6$,\\
$[e_1,e_6] = e_7, [e_2,e_5] = e_7, [e_3,e_4] = -e_7, [e_2,e_6] = -e_7$,\\
  $\eta = e^7$, \quad
$d\eta = -(e^1\wedge e^6 + e^2\wedge e^5 -e^3\wedge e^4 -e^2\wedge e^6)$
\end{tabular} \\
\hline

$\g_{10,2}$	&
\begin{tabular}{l}
$[e_1,e_2] =e_4, [e_1,e_4] =e_5, [e_1,e_3] =e_6, [e_2,e_4] =e_6,$\\
$[e_1,e_6] = -e_7, [e_2,e_5] = -e_7, [e_3,e_4] = e_7, [e_2,e_6] = e_7$,\\
  $\eta = e^7$, \quad
$d\eta = e^1\wedge e^6 + e^2\wedge e^5 -e^3\wedge e^4 -e^2\wedge e^6$
\end{tabular} \\
\hline

$\g_{12}$ &
\begin{tabular}{l}
$[e_1,e_2] =e_4, [e_1,e_4] =e_5, [e_1,e_3] =e_6, [e_2,e_3] = -e_5, [e_2,e_4] =e_6$, \\
$[e_1,e_5] = \lambda e_7, [e_2,e_6] = e_7, [e_3,e_4] = (\lambda +1)e_7$,\\
  $\eta = e^7$, \quad
$d\eta = -(\lambda e^1\wedge e^5 + e^2\wedge e^6 +(\lambda +1) e^3\wedge e^4)$
\end{tabular} \\
\hline

$\g_{24,1}$ &
\begin{tabular}{l}
$[e_1,e_4] =e_6, [e_2,e_3] =e_5$, \\
$[e_1,e_6] = e_7, [e_2,e_5] = e_7, [e_3,e_4] = e_7$,\\
  $\eta = e^7$, \quad
$d\eta = -(e^1\wedge e^6 + e^2\wedge e^5 + e^3\wedge e^4)$
\end{tabular} \\
\hline

$\g_{24,2}$	&
\begin{tabular}{l}
$[e_1,e_4] =e_6, [e_2,e_3] =e_5$,\\
$[e_1,e_6] = -e_7, [e_2,e_5] = -e_7, [e_3,e_4] = -e_7$,\\
  $\eta = e^7$, \quad
$d\eta = e^1\wedge e^6 + e^2\wedge e^5 + e^3\wedge e^4$
\end{tabular} \\
\hline

$\g_{17}$ &
\begin{tabular}{l}
$[e_1,e_3] =e_5, [e_1,e_4] =e_6, [e_2,e_3] =e_6$, \\
$[e_1,e_6] = e_7, [e_2,e_5] = e_7, [e_3,e_4] = e_7$,\\
  $\eta = e^7$, \quad
$d\eta = -(e^1\wedge e^6 + e^2\wedge e^5 + e^3\wedge e^4)$
\end{tabular} \\
\hline

$\g_{16,1}$ &
\begin{tabular}{l}
$[e_1,e_2] =e_5, [e_1,e_3] =e_6, [e_2,e_4] =e_6, [e_3,e_4] = -e_5$, \\
$[e_1,e_6] = e_7, [e_2,e_3] = e_7, [e_4,e_5] = -e_7$,\\
  $\eta = e^7$, \quad
$d\eta = -(e^1\wedge e^6 + e^2\wedge e^3 - e^4\wedge e^5)$
\end{tabular} \\
\hline

$\g_{16,2}$	&
\begin{tabular}{l}
$[e_1,e_2] =e_5, [e_1,e_3] =e_6, [e_2,e_4] =e_6, [e_3,e_4] = -e_5$, \\
$[e_1,e_6] = e_7, [e_2,e_3] = -e_7, [e_4,e_5] = e_7$,\\
  $\eta = e^7$, \quad
$d\eta = -(e^1\wedge e^6 -e^2\wedge e^3 +e^4\wedge e^5)$
\end{tabular} \\
\hline

$\g_{18,1}$	&
\begin{tabular}{l}
$[e_1,e_2] =e_4, [e_1,e_3] =e_5, [e_2,e_3] =e_6$, \\
$[e_1,e_6] = e_7, [e_2,e_5] = \lambda e_7, [e_3,e_4] = (\lambda -1)e_7$,\\
  $\eta = e^7$, \quad
$d\eta = -(e^1\wedge e^6 + \lambda e^2\wedge e^5 + (\lambda -1) e^3\wedge e^4)$
\end{tabular} \\
\hline

$\g_{18,2}$	&
\begin{tabular}{l}
$[e_1,e_2] =e_4, [e_1,e_3] =e_5, [e_2,e_3] =e_6$, \\
$[e_1,e_5] = e_7, [e_1,e_6] = \lambda e_7, [e_2,e_5] = -\lambda e_7, [e_2,e_6] = e_7, [e_3,e_4] = -2\lambda e_7$,\\
  $\eta = e^7$, \quad
$d\eta = -(e^1\wedge e^5 + \lambda e^1\wedge e^6 -\lambda e^2\wedge e^5 + e^2\wedge e^6 -2\lambda e^3\wedge e^4)$
\end{tabular} \\
\hline

$\g_{18,3}$	&
\begin{tabular}{l}
$[e_1,e_2] =e_4, [e_1,e_3] =e_5, [e_2,e_3] =e_6$, \\
$[e_1,e_6] = -e_7, [e_2,e_5] = e_7, [e_3,e_4] = 2e_7, [e_3,e_5] = e_7$,\\
  $\eta = e^7$, \quad
$d\eta = -(-e^1\wedge e^6 + e^2\wedge e^5 + 2e^3\wedge e^4 + e^3\wedge e^5)$
\end{tabular} \\
\hline

$\g_{23,1}$	&
\begin{tabular}{l}
$[e_1,e_2] =e_5, [e_1,e_3] =e_6$, \\
$[e_1,e_4] = e_7, [e_2,e_6] = e_7, [e_3,e_5] = -e_7$,\\
  $\eta = e^7$, \quad
$d\eta = -(e^1\wedge e^4 + e^2\wedge e^6 - e^3\wedge e^5)$
\end{tabular} \\
\hline

$\g_{25}$ &
\begin{tabular}{l}
$[e_1,e_2] =e_6$, \\
$[e_1,e_6] = e_7, [e_2,e_5] = e_7, [e_3,e_4] = e_7$,\\
  $\eta = e^7$, \quad
$d\eta = -(e^1\wedge e^6 + e^2\wedge e^5 + e^3\wedge e^4)$
\end{tabular} \\
\hline
\end{longtable}
\end{theorem}

\begin{proof}
It is well known that the Heisenberg algebra $\h_{2n+1}$ is a central extension of the symplectic abelian Lie algebra, which admits the K\"{a}hler metric. Therefore, $\h_{2n+1}$ admits Sasaki structure. According to Corollary 2 of the previous section, pseudo-Sasakian structures are central extensions of pseudo-K\"{a}hler Lie algebra $(\h,\omega,J)$.
The authors of \cite{CFU2} found all six-dimensional nilpotent Lie algebras which admit pseudo-K\"{a}hler structure. In \cite{Goze-Khakim-Med}, the classification of symplectic six-dimensional nilpotent Lie algebras up to symplecto-isomorphisms is obtained.
The author of \cite{Smolen-11} found all the symplectic Lie algebras on the listing by Goze, Khakimdjanov and Medina which admit compatible complex structures, i.e., which are pseudo-K\"{a}hler. The Lie algebras of the list of the theorem are a central extension of the pseudo-K\"{a}hler Lie algebras found in \cite{Smolen-11}, based on the classification by Goze, Khakimdjanov and Medina. The indices of the Lie algebras correspond to their numbers in the classification list \cite{Goze-Khakim-Med}, and their order corresponds to their properties specified in \cite{Smolen-11}.
\end{proof}

Since the contact metric structure $(\eta, \xi, \phi, g)$ is a central extension $\g =\h\times_\omega\R$ of almost pseudo-K\"{a}hler structures on $\h$, then it is clear that a seven-dimensional nilpotent Lie algebra cannot be (pseudo) Sasakian when the symplectic algebra $(\h,\omega,J)$ is not (pseudo) K\"{a}hler, i.e., does not admit a compatible integrable almost complex structure $J$.
Since the complete list of the six-dimensional nilpotent symplectic Lie algebras admitting compatible complex structure $J$ is given in \cite{CFU2}, then the remaining 12 classes of symplectic six-dimensional nilpotent Lie algebras listed in \cite{Goze-Khakim-Med} do not admit Sasaki structures even with the pseudo-metric.
In addition, many of the Lie algebras in \cite{CFU2} can have several non-isomorphic symplectic structures \cite{Goze-Khakim-Med}. Moreover, some symplectic forms of the same Lie algebra allow compatible complex structures, while others do not. The central extensions of the last Lie algebras also do not admit pseudo-Sasakian structures.
Therefore, we have the following theorem in which the central extensions of symplectic Lie algebras are listed first which do not allow complex structures in general, and then the central extension of symplectic Lie algebras is specified, which may have complex structures but which do not allow the compatible complex structures of the considered symplectic form.

\begin{theorem} \label{Th-2}
The following nilpotent Lie algebras admit $K$-contact structure, but do not admit Sasaki structures even with pseudo-Riemannian metric:
\begin{longtable}[H]{|l|l|}
\hline

$\g_{1}$ &
\begin{tabular}{l}
$[e_1,e_2] = e_3, [e_1,e_3] =e_4, [e_1,e_4] =e_5, [e_1,e_5] =e_6, [e_2,e_3] =e_5,  [e_2,e_4] =e_6$, \\
$[e_1,e_6] = e_7, [e_2,e_5] = (1-\lambda) e_7, [e_3,e_4] = \lambda e_7$, \\
  $\eta = e^7$, \quad
$d\eta = -(e^1\wedge e^6 + (1-\lambda)e^2\wedge e^5 + \lambda e^3\wedge e^4)$\\
\end{tabular} \\
\hline

$\g_{2}$ &
\begin{tabular}{l}
$[e_1,e_2] = e_3, [e_1,e_3] =e_4, [e_1,e_4] =e_5, [e_1,e_5] =e_6, [e_2,e_3] =e_6$, \\
$[e_1,e_6] = \lambda e_7, [e_2,e_4] = \lambda e_7, [e_3,e_4] = \lambda e_7, [e_2,e_5] = -\lambda e_7$, \\
  $\eta = e^7$, \quad
$d\eta = -\lambda (e^1\wedge e^6 + e^2\wedge e^4 + e^3\wedge e^4 - e^2\wedge e^5)$
\end{tabular} \\
\hline

$\g_{3}$ &
\begin{tabular}{l}
$[e_1,e_2] = e_3, [e_1,e_3] =e_4, [e_1,e_4] =e_5, [e_1,e_5] =e_6$, \\
$[e_1,e_6] = e_7, [e_2,e_5] = -e_7, [e_3,e_4] = e_7$, \\
  $\eta = e^7$, \quad
$d\eta = -(e^1\wedge e^6 -e^2\wedge e^5 + e^3\wedge e^4)$
\end{tabular} \\
\hline

$\g_{4}$ &
\begin{tabular}{l}
$[e_1,e_2] = e_3, [e_1,e_3] =e_4, [e_1,e_4] =e_6, [e_2,e_3] =e_5, [e_2,e_5] =e_6$, \\
$[e_1,e_4] = \lambda_1e_7, [e_1,e_5] = \lambda_2e_7, [e_1,e_6] = \lambda_2e_7, [e_2,e_4] = \lambda_2e_7, [e_3,e_5] = \lambda_2e_7$, \\
  $\eta = e^7$, \quad
$d\eta = -\lambda_1e^1\wedge e^4 - \lambda_2(e^1\wedge e^5 + e^1\wedge e^6 + e^2\wedge e^4 + e^3\wedge e^5)$
\end{tabular} \\
\hline

$\g_{5,1}$ &
\begin{tabular}{l}
$[e_1,e_2] = e_3, [e_1,e_3] =e_4, [e_1,e_4] = -e_6, [e_2,e_3] =e_5, [e_2,e_5] =e_6$, \\
$[e_1,e_4] = \lambda_1e_7, [e_1,e_5] = \lambda_2e_7, [e_1,e_6] = \lambda_2e_7, [e_2,e_4] = \lambda_2e_7, [e_3,e_5] = \lambda_2e_7$, \\
  $\eta = e^7$, \quad
$d\eta = -\lambda_1 e^1\wedge e^4 -\lambda_2(e^1\wedge e^5 + e^1\wedge e^6 + e^2\wedge e^4 + e^3\wedge e^5)$
\end{tabular} \\
\hline

$\g_{5,2}$ &
\begin{tabular}{l}
$[e_1,e_2] = e_3, [e_1,e_3] =e_4, [e_1,e_4] = -e_6, [e_2,e_3] =e_5, [e_2,e_5] =e_6$, \\
$[e_1,e_6] = \lambda e_7, [e_1,e_5] = -2\lambda e_7, [e_2,e_4] = -2\lambda e_7, [e_2,e_6] = \lambda e_7, [e_3,e_4] = \lambda e_7$, \\
$[e_3,e_5] = \lambda e_7$, \\
$\eta = e^7$,
$d\eta = -\lambda (e^1\wedge e^6 -2e^1\wedge e^5 -2e^2\wedge e^4 + e^2\wedge e^6 + e^3\wedge e^4 + e^3\wedge e^5)$
\end{tabular} \\
\hline

$\g_{5,3}$ &
\begin{tabular}{l}
$[e_1,e_2] = e_3, [e_1,e_3] =e_4, [e_1,e_4] = -e6, [e_2,e_3] =e_5, [e_2,e_5] =e_6$, \\
$[e_1,e_4] = \lambda e_7, [e_1,e_5] = -\lambda e_7, [e_1,e_6] = \lambda e_7, [e_2,e_4] = -\lambda e_7$, \\
$[e_2,e_5] = \lambda e_7, [e_2,e_6] = \lambda e_7, [e_3,e_4] = \lambda e_7, [e_3,e_5] = \lambda e_7$, $\eta = e^7$,\\
$d\eta = -\lambda  (e^1\wedge e^4 -e^1\wedge e^5 + e^1\wedge e^6 -e^2\wedge e^4 + e^2\wedge e^5 + e^2\wedge e^6 + e^3\wedge e^4 + e^3\wedge e^5)$
\end{tabular} \\
\hline

$\g_{5,4}$ &
\begin{tabular}{l}
$[e_1,e_2] = e_3, [e_1,e_3] =e_4, [e_1,e_4] = -e_6, [e_2,e_3] =e_5, [e_2,e_5] =e_6$, \\
$[e_1,e_4] = 2\lambda e_7, [e_1,e_6] = \lambda e_7, [e_2,e_5] = 2\lambda e_7, [e_2,e_6] = \lambda e_7, [e_3,e_4] = \lambda e_7, [e_3,e_5] = \lambda e_7$, \\
  $\eta = e^7$, \quad
$d\eta = -\lambda (2e_1\wedge e_4  + e^1\wedge e^6  + 2e^2\wedge e^5 + e^2\wedge e^6 + e^3\wedge e^4 + e^3\wedge e^5)$
\end{tabular} \\
\hline

$\g_{6,1}$ &
\begin{tabular}{l}
$[e_1,e_2] = e_3, [e_1,e_3] =e_4, [e_1,e_4] =e_5, [e_2,e_3] =e_6$, \\
$[e_1,e_6] = e_7, [e_2,e_4] = e_7, [e_2,e_5] = e_7, [e_3,e_4] = -e_7$, \\
  $\eta = e^7$, \quad
$d\eta = -(e^1\wedge e^6 + e^2\wedge e^4 + e^2\wedge e^5 -e^3\wedge e^4)$
\end{tabular} \\
\hline

$\g_{6,1}$ &
\begin{tabular}{l}
$[e_1,e_2] = e_3, [e_1,e_3] =e_4, [e_1,e_4] =e_5, [e_2,e_3] =e_6$, \\
$[e_1,e_6] = -e_7, [e_2,e_4] = -e_7, [e_2,e_5] = -e_7, [e_3,e_4] = e_7$, \\
  $\eta = e^7$, \quad
$d\eta = e^1\wedge e^6 + e^2\wedge e^4 + e^2\wedge e^5 -e^3\wedge e^4$
\end{tabular} \\
\hline

$\g_{7,1}$ &
\begin{tabular}{l}
$[e_1,e_2] =e_4, [e_1,e_4] =e_5, [e_1,e_5] =e_6, [e_2,e_3] =e_6, [e_2,e_4] =e_6$, \\
$[e_1,e_3] = \lambda e_7, [e_2,e_6] = \lambda e_7, [e_4,e_5] = -\lambda e_7$, \\
  $\eta = e^7$, \quad
$d\eta = -\lambda (e^1\wedge e^3 + e^2\wedge e^6 -e^4\wedge e^5)$
\end{tabular} \\
\hline

$\g_{7,2}$ &
\begin{tabular}{l}
$[e_1,e_2] =e_4, [e_1,e_4] =e_5, [e_1,e_5] =e_6, [e_2,e_3] =e_6, [e_2,e_4] =e_6$, \\
$[e_1,e_6] = e_7, [e_2,e_5] = e_7, [e_3,e_4] = -e_7$, \\
  $\eta = e^7$, \quad
$d\eta = -(e^1\wedge e^6 + e^2\wedge e^5 -e^3\wedge e^4)$
\end{tabular} \\
\hline

$\g_{7,3}$ &
\begin{tabular}{l}
$[e_1,e_2] =e_4, [e_1,e_4] =e_5, [e_1,e_5] =e_6, [e_2,e_3] =e_6, [e_2,e_4] =e_6$, \\
$[e_1,e_6] = -e_7, [e_2,e_5] = -e_7, [e_3,e_4] = e_7$, \\
  $\eta = e^7$, \quad
$d\eta  = e^1\wedge e^6 + e^2\wedge e^5 - e^3\wedge e^4$
\end{tabular} \\
\hline

$\g_{8}$ &
\begin{tabular}{l}
$[e_1,e_3] =e_4, [e_1,e_4] =e_5, [e_1,e_5] =e_6, [e_2,e_3] =e_5, [e_2,e_4] =e_6$, \\
$[e_1,e_6] = e_7, [e_2,e_5] = e_7, [e_3,e_4] = -e_7$, \\
  $\eta = e^7$, \quad
$d\eta = -(e^1\wedge e^6 + e^2\wedge e^5 -e^3\wedge e^4)$
\end{tabular} \\
\hline

$\g_{9}$ &
\begin{tabular}{l}
$[e_1,e_2] =e_4, [e_1,e_4] =e_5, [e_1,e_5] =e_6, [e_2,e_3] =e_6$, \\
$[e_1,e_3] = \lambda e_7, [e_2,e_6] = \lambda e_7, [e_4,e_5] = -\lambda e_7$, \\
  $\eta = e^7$, \quad
$d\eta = -\lambda (e^1\wedge e^3 + e^2\wedge e^6 -e^4\wedge e^5)$
\end{tabular} \\
\hline

$\g_{19}$ &
\begin{tabular}{l}
$[e_1,e_2] =e_4, [e_1,e_4] =e_5, [e_1,e_5] =e_6$, \\
$[e_1,e_3] = e_7, [e_2,e_6] = e_7, [e_4,e_5] = -e_7$, \\
  $\eta = e^7$, \quad
$d\eta = -(e^1\wedge e^3 + e^2\wedge e^6 -e^4\wedge e^5)$
\end{tabular} \\
\hline

$\g_{20}$ &
\begin{tabular}{l}
$[e_1,e_2] = e_3, [e_1,e_3] =e_4, [e_1,e_4] =e_5, [e_2,e_3] =e_5$, \\
$[e_1,e_6] = e_7, [e_2,e_5] = e_7, [e_3,e_4] = -e_7$, \\
  $\eta = e^7$, \quad
$d\eta = -(e^1\wedge e^6 + e^2\wedge e^5 -e^3\wedge e^4)$
\end{tabular} \\
\hline

$\g_{22}$ &
\begin{tabular}{l}
$[e_1,e_2] =e_5, [e_1,e_5] =e_6$, \\
$[e_1,e_6] = e_7, [e_2,e_5] = e_7, [e_3,e_4] = e_7$, \\
  $\eta = e^7$, \quad
$d\eta = -(e^1\wedge e^6 + e^2\wedge e^5 + e^3\wedge e^4)$
\end{tabular} \\
\hline
		
$\g_{14,2}$ &
\begin{tabular}{l}
$[e_1,e_2] =e_4, [e_2,e_3] =e_6, [e_2,e_4] =e_5$, \\
$[e_1,e_4] = -e_7, [e_2,e_5] = e_7, [e_3,e_6] = e_7$, \\
  $\eta = e^7$, \quad
$d\eta = -(-e^1\wedge e^4 + e^2\wedge e^5 + e^3\wedge e^6)$
\end{tabular} \\
\hline

$\g_{14,3}$ &
\begin{tabular}{l}
$[e_1,e_2] =e_4, [e_2,e_3] =e_6, [e_2,e_4] =e_5$, \\
$[e_1,e_4] = e_7, [e_2,e_5] = e_7, [e_3,e_6] = e_7$, \\
  $\eta = e^7$, \quad
$d\eta = -(e^1\wedge e^4 + e^2\wedge e^5 + e^3\wedge e^6)$
\end{tabular} \\
\hline

$\g_{21,3}$ &
\begin{tabular}{l}
$[e_1,e_2] =e_4, [e_1,e_4] =e_6, [e_2,e_3] =e_6$, \\
$[e_1,e_6] = e_7, [e_2,e_4] = e_7, [e_3,e_4] = -e_7, [e_3,e_5] = -e_7$, \\
  $\eta = e^7$, \quad
$d\eta = -(e^1\wedge e^6 +e^2\wedge e^4 -e^3\wedge e^4 -e^3\wedge e^5)$
\end{tabular} \\
\hline

$\g_{13,3}$ &
\begin{tabular}{l}
$[e_1,e_2] =e_4, [e_1,e_3] =e_5, [e_1,e_4] =e_6, [e_2,e_3] =e_6$, \\
$[e_1,e_6] = e_7, [e_2,e_4] = \lambda e_7, [e_2,e_5] = e_7, [e_3,e_5] = e_7$, \\
  $\eta = e^7$, \quad
$d\eta = -(e^1\wedge e^6 + \lambda e^2\wedge e^4 + e^2\wedge e^5 + e^3\wedge e^5)$
\end{tabular} \\
\hline

$\g_{15,3}$ &
\begin{tabular}{l}
$[e_1,e_2] =e_4, [e_1,e_4] =e_6, [e_2,e_3] =e_5$, \\
$[e_1,e_6] = e_7, [e_2,e_4] = e_7, [e_3,e_5] = e_7$, \\
  $\eta = e^7$, \quad
$d\eta = -(e^1\wedge e^6 + e^2\wedge e^4 + e^3\wedge e^5)$
\end{tabular} \\
\hline

$\g_{23,2}$ &
\begin{tabular}{l}
$[e_1,e_2] =e_5, [e_1,e_3] =e_6$, \\
$[e_1,e_6] = e_7, [e_2,e_5] = e_7, [e_3,e_4] = e_7$, \\
  $\eta = e^7$, \quad
$d\eta = -(e^1\wedge e^6 + e^2\wedge e^5 + e^3\wedge e^4)$
\end{tabular} \\
\hline

$\g_{23,3}$ &
\begin{tabular}{l}
$[e_1,e_2] =e_5, [e_1,e_3] =e_6$, \\
$[e_1,e_4] = e_7, [e_2,e_6] = e_7, [e_3,e_5] = e_7$, \\
  $\eta = e^7$, \quad
$d\eta = -(e^1\wedge e^4 + e^2\wedge e^6 + e^3\wedge e^5)$
\end{tabular} \\
\hline

\end{longtable}
\end{theorem}

\subsection{Covariant derivatives and curvature} \label{Formuly}
In \cite{Smolen-11} it is shown that the pseudo-K\"{a}hler structure on a six-dimensional nilpotent Lie algebra has a zero Ricci tensor. However, for Sasaki structures the Ricci tensor is non-zero, even if such a structure is obtained from a pseudo-K\"{a}hler Lie algebra by central extension. Therefore, in this section we establish the formulas, the binding properties of the curvature of almost pseudo-K\"{a}hler Lie algebras and the contact structures obtained by central extensions.

\begin{lemma} \label{Lem-Covar}
Let $(\omega, J, h)$ be a almost (pseudo) K\"{a}hler structure on the Lie algebra $\h$ and $(\eta, \xi, \phi, g) $ be the corresponding contact metric structure on the central extension  $\g =\h\times_\omega\R$. Then the covariant derivative $\nabla$ on $\g$ expressed in terms of the covariant derivative {\rm D} on $\h$, symplectic form $\omega$ and almost complex structure $J$ on $\h$ are as follows:
$$
\nabla_X Y = {\rm D}_XY  + \frac 12 \omega(X,Y)\xi,
$$
$$
\nabla_X \xi = \nabla_\xi X = -\frac 12 JX  \mbox{ и } \nabla_\xi \xi = 0.
$$
where $X,Y \in \h$.
\end{lemma}

\begin{proof}
Direct calculations using a six-membered formula \cite{KN}, which is for the left-invariant vector fields $X,Y,Z$ on the Lie group takes the form: $2g(\nabla_XY,Z) = g([X,Y],Z) +g([Z,X],Y) +g(X,[Z,Y])$. Take two left-invariant vector fields of the form $x = X + \lambda\xi$ and $y = Y + \mu\xi$, where $X,Y \in \h$.  Then, using the formula $[X, Y]_{\g} = [X,Y]_{\h} +\omega(X, Y)\xi$, $[X,\xi]_{\g} =0$ and orthogonality $\h$ to $\xi$, we obtain for $X,Y \in \h$:
$$
2g(\nabla_XY,Z) = h([X,Y]_{\h},Z) +h([Z, X]_{\h},Y) +h(X,[Z,Y]_{\h}) = 2h({\rm D}_XY,Z),
$$
$$
2g(\nabla_XY,\xi) = g([X,Y]_{\g}, \xi) + g[\xi, X]_{\g},Y) + g(X,[\xi,Y]_{\g}) = \omega(X,Y).
$$
Consequently, $\nabla_X Y = {\rm D}_XY  + \frac 12 \omega(X, Y)\xi$. Further, considering $h(X,Y) = \omega(X, JY)$:
\begin{multline*}
2g(\nabla_X \xi,Z) = g([X,\xi]_{\g},Z) +g[Z,X]_{\g},\xi) +g(X,[Z,\xi]_{\g}) =\omega(Z,X) = \\
=-\omega(Z,JJX) =  -h(Z,JX) = -g(JX, Z),
\end{multline*}
$$
2g(\nabla_X \xi,\xi) = g([X,\xi]_{\g},\xi) +g[\xi, X]_{\g},\xi) +g(X,[\xi,\xi]_{\g}) =0.
$$
Therefore, $\nabla_X \xi =-\frac 12 JX \in \h$. Similarly, we obtain:
$\nabla_\xi X  =-\frac 12 JX$.
\end{proof}

\begin{theorem} \label{Th-3-Riem}
Let $(\omega, J, h)$ be a almost (pseudo) K\"{a}hler structure on the Lie algebra $\h$ and $(\eta, \xi, \phi, g)$ be the corresponding contact metric structure on the central extension  $\g =\h\times_\omega\R$. Then the curvature tensor $R$ of $\g$ expressed in terms of the curvature tensor $R_{\h}$ of $\h$, symplectic form $\omega$ and almost complex structure $J$ on $\h$ are given by the formulas:
$$
R(X,Y)Z = R_{\h}(X,Y)Z -\frac 12 ({\rm D}_Z\omega)(X,Y) \xi -\frac 14 (\omega(Y, Z) JX  - \omega(X, Z) JY ) + \frac 12 \omega(X, Y)JZ,
$$
$$
R(X,Y)\xi = -\frac 12 (({\rm D}_X J)Y -({\rm D}_Y J)X),
$$
$$
R(X, \xi)Z = -\frac 12 ({\rm D}_XJ)Z -\frac 14  g(X, Z)\xi, \qquad  R(X, \xi) \xi = \frac 14  X,
$$
where $X,Y \in \h$.
\end{theorem}

\begin{proof}
This consists of direct calculation using the formula $R(X,Y)Z = \nabla_X \nabla_YZ -\nabla_Y\nabla_XZ - \nabla_{[X,Y]}Z$ using formulas for covariant derivatives obtained in Lemma \ref{Lem-Covar}. For example, let us take first the left-invariant vector field of the form $X,Y,Z \in \h$. Then, using the formula  $[X,Y]_{\g} =[X,Y]_{\h} +\omega(X,Y)\xi$, $[X,\xi]_{\g} =0$, orthogonality $\h$ to $\xi$, and $\nabla_X Y = {\rm D}_XY  +\frac 12 \omega(X,Y)\xi$, we obtain:
\begin{multline*}
R(X,Y)Z = \nabla_X \nabla_YZ -\nabla_Y\nabla_XZ -\nabla_{[X,Y]}Z = \\
= \nabla_X \{{\rm D}_YZ + \frac 12 \omega(Y,Z)\xi\} -\nabla_Y\{{\rm D}_XZ +\frac 12  \omega(X, Z)\xi\} -\nabla_{[X,Y]_{\h}+ \omega(X,Y)\xi} Z =\\
= {\rm D}_X{\rm D}_YZ + \frac 12 \omega(X,{\rm D}_YZ) \xi + \frac 12 \omega(Y, Z)(-\frac 12 JX) -\{{\rm D}_Y{\rm D}_XZ +\frac 12 \omega(Y,{\rm D}_XZ)\xi +\frac 12 \omega(X,Z)(-\frac 12 JY)\} -\\
 -\{{\rm D}_{[X,Y]} Z +\frac 12 \omega([X,Y],Z)\xi\} -\omega(X,Y)\nabla_\xi Z  =  R_{\h}(X,Y)Z + \\
+ \frac 12 \{\omega(X,{\rm D}_YZ) -\omega(Y,{\rm D}_XZ) -\omega([X,Y],Z)\}\xi - \frac 14 \{\omega(Y, Z) JX - \omega(X, Z)JY\} +\frac 12 \omega(X, Y)JZ =
\end{multline*}
To simplify $\omega(X,{\rm D}_YZ) -\omega(Y,{\rm D}_XZ) -\omega([X,Y], Z)$ in the last expression, we use the properties ${\rm D}_XZ = {\rm D}_ZX +[X,Z]$ and ${\rm D}_YZ = {\rm D}_ZY +[Y,Z]$. Then, considering the closedness of the form $\omega$, $\omega(X,[Y,Z]) + \omega(Y,[Z,X]) + \omega( Z,[X,Y]) = d\omega(X,Y,Z) = 0$ and equality $Z\omega(X, Y) = ({\rm D}_Z\omega)(X,Y) + \omega({\rm D}_ZX,Y) + \omega(X,{\rm D}_ZY) = 0$, we obtain:
\begin{multline*}
	\omega(X, {\rm D}_YZ) -\omega(Y, {\rm D}_XZ) -\omega([X,Y],Z) = \\
	= \omega(X,{\rm D}_ZY +[Y,Z]) -\omega(Y,{\rm D}_ZX + [X,Z]) -\omega([X,Y],Z) = \\
	= \omega(X,{\rm D}_ZY) -\omega(Y,{\rm D}_ZX) +\omega(X,[Y,Z]) -\omega(Y,[X,Z]) -\omega([X, Y], Z) = \\
	= \omega({\rm D}_ZX, Y) + \omega(X,{\rm D}_ZY) +\omega(X,[Y,Z]) + \omega(Y,[Z,X]) + \omega(Z,[X,Y]) = -({\rm D}_Z\omega)(X,Y).
\end{multline*}
Substituting this expression, we complete the calculation:
$$
=R_{\h}(X,Y)Z -\frac 12 ({\rm D}_Z\omega)(X,Y)\xi -\frac 14 (\omega(Y,Z)JX  -\omega(X,Z)JY) + \frac 12 \omega(X,Y)JZ.
$$
Consider the case of $R(X,Y)\xi$, where $X,Y \in \h$. Then, considering equalities $\nabla_X \xi = -\frac 12 JX$ and $\nabla_\xi \xi = 0$ we get:
\begin{multline*}
R(X,Y)\xi = \nabla_X (\nabla_Y \xi) -\nabla_Y (\nabla_X \xi) -\nabla_{[X,Y]} \xi = -\nabla_X (\frac 12 JY) + \nabla_Y (\frac 12 JX)  -\nabla_{[X,Y]_{\h}+ \omega(X, Y)\xi} \xi =\\
= \frac 12 (-\nabla_X (JY) + \nabla_Y (JX)  +J[X, Y]) = \\
= \frac 12 (-{{\rm D}_X JY -\frac 12 \omega(X, JY)\xi} +{{\rm D}_Y JX +\frac 12 \omega(Y, JX)\xi } + J[X, Y]) =\\
= \frac 12 ( -({\rm D}_X J)Y -J({\rm D}_XY) -\frac 12 \omega(X, JY)\xi +
({\rm D}_Y J)X +J({\rm D}_YX) +\frac 12 \omega(Y, JX)\xi +J[X, Y] ) =\\
= \frac 12(-({\rm D}_X J)Y + ({\rm D}_Y J)X -J({\rm D}_XY-{\rm D}_YX) -\frac 12 \omega(X, JY)\xi +\frac 12 \omega(Y, JX)\xi + J[X, Y]) =\\
= \frac 12 (-({\rm D}_X J)Y +({\rm D}_Y J)X -J([X,Y]) -\frac 12 g(X,Y)\xi +\frac 12 g(Y,X)\xi +J[X, Y]) = \\
 = -\frac 12  (({\rm D}_X J)Y -({\rm D}_YJ)X).
\end{multline*}
Now consider the case $R(X, \xi)Y$, where $X,Y \in \h$. Similarly, we obtain:
\begin{multline*}
R(X, \xi)Z = \nabla_X(\nabla_\xi Z) -\nabla_\xi(\nabla_XZ) - \nabla_{[X, \xi]} Z = \\
= -\nabla_X (\frac 12 JZ) -\nabla_\xi ({\rm D}_XZ + \frac 12 \omega(X, Z)\xi) = -\nabla_X (\frac 12 JZ) + \frac 12 J({\rm D}_XZ) = \\
= \frac 12 (-{{\rm D}_XJZ -\frac 12 \omega(X, JZ)\xi} +J({\rm D}_XZ)) = \frac 12 (-({\rm D}_XJ)Z - J({\rm D}_XZ) - \frac 12 \omega(X,JZ)\xi +J({\rm D}_XZ)) = \\
= -\frac 12 ({\rm D}_XJ)Z  -\frac 14 \omega(X,JZ)\xi = -\frac 12 ({\rm D}_XJ)Z  -\frac 14 g(X, Z)\xi.
\end{multline*}
Similarly, established the last formula,
$$
R(X, \xi) \xi = \nabla_X(\nabla_\xi \xi) -\nabla_\xi(\nabla_X \xi) -\nabla_{[X, \xi]}\xi =  -\nabla_\xi(-\frac 12 JX) = \frac 12 \nabla_\xi (JX) = -\frac 12 \frac 12 J(JX)  = \frac 14 X.
$$
\end{proof}

In the case of a (pseudo) K\"{a}hler structure on the Lie algebra $\h$ we have $D\omega = 0$ and $DJ = 0$. Therefore, the formulas for the curvature will have a simpler form. If additionally submit an expression of $\omega(Y,Z)JX  -\omega(X,Z)JY$ as $h(Z,JY)JX -h(Z,JX)JY$, we get:

\begin{corollary} \label{Cor-3}
Let $(\omega, J, h)$ be a (pseudo) K\"{a}hler complex structure on the Lie algebra $\h$ and $(\eta, \xi, \phi, g)$ be the corresponding (pseudo) Sasakian structure at the central expansion $\g =\h\times_\omega\R$. Then the curvature tensor $R$ on $\g$ is expressed in terms of the curvature tensor $R_{\h}$ on $\h$, the symplectic form $\omega$  and an almost complex structure $J$ on $\h$ as follows:
$$
R(X,Y)Z = R_{\h}(X,Y)Z -\frac 14 (h(Z,JY)JX -h(Z,JX)JY) +\frac 12 \omega(X,Y)JZ,
$$
$$
R(X,Y)\xi =0, \qquad R(X,\xi)Z =-\frac 14 g(X,Z)\xi, \qquad R(X,\xi)\xi =\frac 14 X,
$$
where $X,Y \in \h$.
\end{corollary}

\subsection{Ricci tensor} \label{Formuly-Ricci}
Recall that the Ricci tensor $Ric$ in a pseudo-Riemannian case is defined by the formula:
$$
Ric(X,Y) = \sum_{i=1}^{2n+1} \varepsilon_i g(R(e_i,Y)Z, e_i),
$$
where $\{e_i\}$ is the orthonormal basis on $\g$ and $\varepsilon_i =g(e_i, e_i)$. We choose a basis for $\g$ in the form $\{e_1, \dots , e_{2n}, e_{2n+1}\} =\{E_1, \dots, E_{2n}, \xi \}$, where $E_i \in \h$ and $\xi$ is the Reeb field. The following calculations assume that the index $i$ changes from 1 to $2n+1$, and the index $j$ changes from 1 to $2n$. In addition, we consider that an almost complex structure on $\h$ is integrable, so that $\h$ is the (pseudo) K\"{a}hler Lie algebra.

\begin{theorem} \label{Th-4}
Let $(\omega, J, h)$ be (pseudo) K\"{a}hler structure on the Lie algebra $\h$, and $(\eta, \xi, \phi, g)$ corresponds to a contact (pseudo) metric Sasaki structure on a central expansion $\g =\h\times_\omega\R$. Then the Ricci tensor on $\g$ expressed by the Ricci tensor $Ric_{\h}$ on $\h$ forms $\omega$ and an almost complex structure $J$ on $\h$ as follows:
$$
Ric(Y,Z) = Ric_{\h}(Y,Z) -\frac 12 h(Y,Z),
$$
$$
Ric(Y, \xi) = 0,\qquad Ric(\xi,\xi) = n/2.
$$
where $X,Y \in \h$.
\end{theorem}

\begin{proof}
In the basis $\{e_1, \dots , e_{2n}, e_{2n+1}\} = \{E_1, \dots, E_{2n}, \xi \}$ of the algebra $\g$, where $E_j \in \h$, for the $Y,Z \in \h$ we obtain:
\begin{multline*}
Ric(Y,Z) = \sum_i \varepsilon_i g(R(e_i,Y)Z, e_i) = \\
= \sum_j \left(\varepsilon_j h(R_{\h}(E_j,Y)Z, E_j) - \frac 14 \varepsilon_i h(\omega(Y, Z)JE_j -\omega(E_j, Z) JY, E_j) + \frac 12  \varepsilon_j \omega(E_j,Y)h(E_j,JZ)\right) + \\
+g(R(\xi,Y)Z, \xi) =\\
= Ric_{\h}(Y,Z) -\frac 14 \sum_j \varepsilon_j \omega(Y, Z)h(JE_j,E_j) +\frac 14 \sum_j \varepsilon_j \omega(E_j,Z) h(JY,E_j) +\\
 + \frac 12  \sum_j \varepsilon_j \omega(E_j,-JJY)g(E_j,JZ) + g(\frac 14  g(Y, Z)\xi,\xi)  = \\
= Ric_{\h}(Y,Z) + \frac 14  \sum_j \varepsilon_j \omega(E_j, -JJZ)h(JY,E_j) +\frac 14  h(Y, Z) -\frac 12  \sum_j \varepsilon_j g(E_j,JY)g(E_j,JZ) =\\
 = Ric_{\h}(Y,Z) - \frac 14 \sum_j \varepsilon_j h(E_j, JZ) h(E_j, JY) + \frac 14  h(Y, Z) - \frac 12  \sum_j \varepsilon_j g(E_j, JY)g(E_j,JZ) = \\
= Ric_{\h}(Y,Z) -\frac 14 h(JY, JZ) + \frac 14  h(Y, Z) -\frac 12  h(JY, JZ) =  Ric_{\h}(Y,Z) -\frac 12  h(Y, Z).
\end{multline*}
Further,
$$
Ric(Y, \xi) = \sum_i g(R(e_i,Y)\xi, e_i) = \sum_j \varepsilon_j g(R(E_j,Y)\xi, E_j) + g(R(\xi,Y)\xi, \xi) = -\frac 14  g(Y, \xi) = 0.
$$
\begin{multline*}
Ric(\xi, \xi) = \sum_i \varepsilon_i  g(R(e_i, \xi)\xi, e_i) = \sum_j \varepsilon_j  g(R(E_j, \xi)\xi, E_j) + g(R(\xi, \xi) \xi, \xi) = \\
=\frac 14  \sum_j \varepsilon_j g(E_j, E_j) = \frac 14 \sum_j \varepsilon_j  \varepsilon_j  = n/2.
\end{multline*}
\end{proof}

In \cite{Smolen-11}, the invariant pseudo-K\"{a}hler structures on six-dimensional nilpotent Lie algebra with the symplectic structure of the classification list  \cite{Goze-Khakim-Med} are studied in detail. For each of them there are many compatible complex structures and corresponding pseudo-Riemannian metrics. It was found that they all have common properties: the associated pseudo-Kahler metric is Ricci-flat, the Riemann tensor has zero pseudo-norm, and the Riemann tensor has several nonzero components which depend only on two or, at most, three parameters. The author of \cite{Smolen-11} found a pseudo-K\"{a}hler structure depending only on the parameters that have an impact on curvature. Such metrics are called canonical. The curvature properties of the almost pseudo-K\"{a}hler structures on six nilpotent Lie algebra are obtained in \cite{Smolen-13}. Each (almost) pseudo-K\"{a}hler structure on a six-dimensional nilpotent Lie algebra determines the pseudo-Riemannian ($K$-contact) Sasaki structure on the seven-dimensional nilpotent Lie algebra. The formulas of Theorems \ref{Th-3-Riem} and \ref{Th-4} allow the use of the properties of the pseudo-K\"{a}hler structures for the properties of the corresponding contact ($K$-contact) Sasaki structures. This is demonstrated by the following example.

\textbf{Example}\\
Consider the Lie algebra $\g_{14}$ with the commutators $[e_1,e_2] =e_4, [e_2,e_3] =e_6, [e_2,e_4] =e_5$, $[e_1,e_6] = -e_7, [e_2,e_5] = e_7, [e_3,e_4] = e_7$ and contact form $\eta = e^7$.
It is central extension of the algebra $\h_{14}$, $[e_1,e_2] =e_4, [e_2,e_3] =e_6, [e_2,e_4] = e_5$, using the symplectic form $\omega = -e^1\wedge e^6 + e^2\wedge e^5 + e^3\wedge e^4$. As noted in the example in the previous section, there is \cite{Smolen-11} a family of pseudo-K\"{a}hler metrics on $\h_{14}$ which depend on six parameters $\psi_{ij}$. The curvature tensor on $\h_{14}$ is zero for all values of the parameters $\psi_{ij}$. Therefore, in this case, we obtain a family of contact pseudo-Sasaki structures $(\eta, \xi, \phi, g)$ on $\g_{14}$, where $\xi = e_7$ and $\phi(x) = JX$, if $x = X + \lambda\xi$, $X \in \h_{14}$. The curvature tensor and Ricci tensor have the form:
$$
R(X,Y)Z = -\frac 14 (g(Z,JY)JX -g(Z,JX)JY) +\frac 12 \omega(X,Y)JZ,
$$
$$
R(X,Y)\xi =0, \quad R(X,\xi) Z = -\frac 14 g(X,Z)\xi,  \quad R(X,\xi)\xi =\frac 14 X,
$$
$$
Ric(Y,Z) = -\frac 12 g(Y,Z), \quad Ric(Y,\xi) = 0,\quad  Ric(\xi,\xi) =n/2.
$$

\textbf{Remark.}
Using the results of \cite{Smolen-11} it is easy to build a family of contact pseudo-Riemannian Sasaki structures and other contact structures from Theorem \ref{Th-1}.
	
In \cite{Goze-Remm} it is shown that the contact nilpotent Lie algebra of dimension $2n+1$ is a central extension of the $2n$-dimensional symplectic nilpotent Lie algebra. Therefore, contact nilpotent Lie algebras in the tables of Theorems \ref{Th-1} and \ref{Th-2} with the Heisenberg algebra give a full list of all contact seven-dimensional nilpotent Lie algebras. Another classification of seven-dimensional contact nilpotent Lie algebras is introduced in \cite{Kutsak}. In contrast to the symplectic case, this paper shows that the Ricci tensor of the contact seven-dimensional Lie algebras is nonzero even in the direction of contact distribution $D$.

\end{document}